\documentclass[11pt,twoside]{amsart}

\numberwithin{equation}{section}
\usepackage{latexsym,amssymb,amsmath,amsthm,amsfonts,amscd}
\usepackage{enumerate}
\usepackage[table]{xcolor}
\usepackage{graphicx}
\usepackage{tikz}
\usepackage{float}
\usetikzlibrary{matrix,arrows}

\definecolor{VerdeOlivo}{rgb}{0.3,0.5,0.1}
\definecolor{Magenta}{rgb}{.65,0.15,.2}
\definecolor{Gris}{gray}{0.3}

\newtheorem{Theorem}{Theorem}[section] 
\newtheorem{Definition}[Theorem]{Definition}
\newtheorem{Proposition}[Theorem]{Proposition}  
 
\newtheorem{Corollary}[Theorem]{Corollary}

\newtheorem{Example}[Theorem]{Example}
\newtheorem{Claim}[Theorem]{Claim}    
 
\newtheorem{Algorithm}[Theorem]{Algorithm}

\theoremstyle{definition}

\def\thinedge{thin}
\def\thickedge{very thick}
\def\colorfill{black!7}

\def\widthii{14}
\def\heightii{0.6}

\begin{document} 


\title[Optimum matchings in weighted bipartite graphs]{Optimum matchings in weighted bipartite graphs}


\author{Carlos E. Valencia}
\author{Marcos C. Vargas}
\address{
Departamento de Matem\'aticas\\
Centro de Investigaci\'on y de Estudios Avanzados del IPN\\
Apartado Postal 14--740\\
07000 Mexico City, D.F. 
} 
\email[Carlos E.~Valencia\footnote{Corresponding author}]{cvalencia@math.cinvestav.edu.mx, cvalencia75@gmail.com}
\email[Marcos C. Vargas]{esocram@gmail.com}
\thanks{Both authors was partially supported by CONACyT grant 166059.
The first author was partially supported by SNI and the second author was partially supported by CONACyT}



\begin{abstract}
Given an integer weighted bipartite graph $\{G=(U\sqcup V, E), w:E\rightarrow \mathbb{Z}\}$ we consider the problems of finding all 
the edges that occur in some minimum weight matching of maximum cardinality and enumerating all the minimum weight perfect matchings. 
Moreover, we construct a subgraph $G_{cs}$ of $G$ which depends on an $\epsilon$-optimal solution of the dual linear program 
associated to the assignment problem on $\{G,w\}$ that allows us to reduced this problems to their unweighed variants on $G_{cs}$.
For instance, when $G$ has a perfect matching and we have an $\epsilon$-optimal solution of the dual linear program 
associated to the assignment problem on $\{G,w\}$, we solve the problem of finding all the edges that occur in some
minimum weight perfect matching in linear time on the number of edges.
Therefore, starting from scratch we get an algorithm that solves this problem in time $O(\sqrt{n}m\log(nW))$, 
where $n=|U|\geq |V|$, $m=|E|$, and $W={\rm max}\{|w(e)|\, :\, e\in E\}$.
\end{abstract}

\maketitle


\section{Introduction} 

Given an {\it integer weighted bipartite graph} $\{G,w\}$, that is, a bipartite graph $G=(U\sqcup V, E)$ with 
bipartitions $U$ and $V$ and an integer weight function $w:E\rightarrow \mathbb{Z}$ over the edges of $G$.
A {\it matching} of $G$ is a set of edges $M$ which are vertex disjoint.
Moreover, a matching is called {\it perfect} if it covers all the vertices of $G$.
A bipartite graph is called {\it feasible} if it has at least one perfect matching.
The {\it weight} of a matching $M$ is given by
\[
w(M)=\sum_{uv\in M}\; w(uv).
\]
A matching of maximum cardinality and minimum weight will be called {\it optimum matching}.
In this paper we reserve the symbols $m=|E|$, $n=|U|\geq |V|$ and $W$ to denote the maximum absolute weight.

Matchings in graphs is one of the most important topics in combinatorial optimization and has been extensively 
studied since the early nineteenth century until now. 
Moreover, it have given origin to several key ideas in combinatorial optimization.
For bipartite graphs with no weights there are algorithms for finding maximum cardinality matchings~\cite{hopcroft, AltBlumMehlhornPaul},
for enumerating all the perfect matchings~\cite{Takeaki2,FukudaMatsui},
for finding all the edges that occur in some maximum cardinality matching~\cite{tassa, costa, regin}, etc.
Finding matchings have many applications in mathematics, computer sciences, operations research, biochemistry, electrical engineering, etc.
For weighted bipartite graphs there are several algorithms that solve the {\it assignment problem}, which consist of finding
a perfect matching of minimum weight, see for instance~\cite{Kuhn, Gabow, DinicKronrod, auctionA_Bert1, gabo_tarj1, Gold_ken2}.
The reader can consult~\cite{burkard, lovasz, line_opt1, sch} and the references therein for more information about assignment problems.

As we mentioned before, several authors have worked on the problem of finding all the edges that occur in some perfect
matching of a bipartite graph and also on the problem of enumerating all the perfect matchings of a bipartite graph.
In this article we solve efficiently the weighted variants of these problems. 
More precisely, we give algorithms for finding all the edges that occur in some minimum weight perfect 
matching and for enumerating all the minimum weight perfect matchings of a weighted bipartite graph $\{G,w\}$.
Moreover, given an optimal solution $P$ of the dual linear program associated to the assignment problem on $\{G,w\}$
we construct a subgraph $G_{cs}(P)$ of $G$, such that if we can solve a problem involving the perfect matchings of $G$, 
then we can use the subgraph $G_{cs}(P)$ to solve the weighted variant of this problem. 
For instance, since we can enumerate all the perfect matchings of a bipartite graph, 
then we can enumerate all the minimum weight perfect matchings of a weighted bipartite graph.
The following result, which will be proven later, provides the relation between the 
optimum perfect matchings of $\{G,w\}$ and the perfect matchings of $G_{cs}(P)$.

{\bf Theorem~\ref{correspondence}}
{\it If $P$ is an optimal solution of the dual linear program associated to the assignment problem on $\{G,w\}$, 
then the set of minimum weight perfect matchings of $\{G,w\}$ is equal to the set of perfect matchings of $G_{cs}(P)$.}

The article is organized as follows:
In Section~\ref{FGopt} we present the linear program that models the assignment problem and its dual. 
We define what is an optimal solution and an $\epsilon$-optimal solution of these linear programs.
Also, we construct the subgraph $G_{cs}(P)$ from an optimal and $\epsilon$-optimal solution $P$ of this dual linear program.
In Section~\ref{applications} we use $G_{cs}(P)$ to finding all the edges that occur in some minimum weight perfect 
matching and enumerating all the minimum weight perfect matchings. 
Also, as an application, we solve the problem of finding a perfect matching where some edges are prefered.
Finally, in section~\ref{FGoptNPM} we give some strategies to deal with these problems when the bipartite
graph has no perfect matchings and we are interested in optimum matchings.

The main result of this paper is based on the following key observation that follows directly from the Complementary Slackness Theorem, 
which can be found in~\cite[Theorem 4.5]{line_opt1}. 
Given a linear program in standard form with decision variables ${\bf x}=(x_1, \ldots, x_r)$, cost vector ${\bf c}=(c_1, \ldots, c_r)$, 
$q\times r$ constraint matrix ${\bf A}$ with columns $A_j$, and dual variables ${\bf p}=(p_1, \ldots, p_q)$. 
Then the optimal primal solutions of this linear program are characterized by the following result:

\begin{Corollary}\label{optimal_face_001}
If ${\bf p}$ is an optimal dual solution, then a feasible primal solution ${\bf x}$ is optimal if and only if $x_j=0$ 
for all  $j$ such that $(c_j- \textbf{p}\cdot A_j)\neq 0$.
\end{Corollary}


\section{The subgraph $G_{cs}$ for feasible bipartite graphs}\label{FGopt}

In this section we define the main object of this paper: the subgraph $G_{cs}$. 
Before we define it, we present the linear program that models the assignment problem, and its dual.
After that, we define what means optimal and $\epsilon$-optimal solutions of the linear assignment program and its dual.

The assignment problem can be modeled by the following linear program, which we call the assignment program.
\begin{equation} \label{line_assig_problem_def_001}
\begin{tabular}{rrl}
minimize:		&	$\displaystyle \sum_{uv\in E} M(uv)w(uv)$	&	\\ 
subject to:	&	$\displaystyle \sum_{v\in N(u)} M(uv) = 1$	&	$\forall u\in U$,  \\ 
		&	$\displaystyle \sum_{u\in N(v)} M(uv) = 1$	&	$\forall v\in V$,  \\
		&	$\displaystyle M(uv)  \geq  0$				&	$\forall uv\in E$, 
\end{tabular}
\end{equation}
where $M$ represents the incidence vector of the matching, see~\cite[Section 7.8]{line_opt1}.
Since the matrix that defines this linear program is the incidence matrix of the bipartite graph $G$,
then it is totally unimodular and therefore we can assume that $M(uv)\in \{0,1\}$ for all $uv\in E$; see~\cite[chapter 18]{sch}.
The dual linear program of~\ref{line_assig_problem_def_001}, that we call dual assignment program, is given by:
\begin{equation}\label{dual_assig_problem_def_001}
\begin{tabular}{rcl}
max:		&	$\displaystyle \sum_{u\in U} \pi(u) + \sum_{v\in V} p(v)$	& \\ 
		&																& \\
subject to:	&	$\displaystyle \pi(u)+p(v) \leq w(uv)$						&	$\forall uv\in E.$  
\end{tabular}
\end{equation}
The dual variables $\pi:U \rightarrow \mathbb{R}$ are associated to the constraints $\sum_{v\in N(u)} M(uv) = 1$ for each $u\in U$ and
the dual variables $p:V \rightarrow \mathbb{R}$ to the constrains $\sum_{u\in N(v)} M(uv) = 1$ for each $v\in V$.
Therefore, the pair $P=(\pi, p)$, also called \emph{dual prices} or simply \emph{prices}, is the set of dual variables of~\ref{dual_assig_problem_def_001}.
Derived from the complementary slackness theorem, we get the following result.

\begin{Proposition}\label{complem_slack_004}
Let $M$ be a perfect matching of $G$ and $P=(\pi,p)$ prices of the dual assignment program.
Then $M$ is a minimum weight perfect matching of $G$, and $P$ are optimal prices of~\ref{dual_assig_problem_def_001}  if and only if
\begin{eqnarray*}
\pi(u)+p(v)&\leq& w(uv) \quad \forall \, uv\in E, \\
\pi(u)+p(v)&=&w(uv) \quad \forall \, uv\in M. 
\ \end{eqnarray*}
\end{Proposition}

Now, let $\epsilon>0$, $M$ be a perfect matching of $G$, and $P=(\pi,p)$ be dual prices of~\ref{dual_assig_problem_def_001}.
We say that $M$ and $P$ are $\epsilon$\emph{-optimal solutions} of the assignment program if and only if
\begin{center}
\begin{tabular}{rcll}
$\pi(u)+p(v)$	& $\leq$	& $w(uv)+\epsilon$	& $\forall \, uv\in E$, \\
$\pi(u)+p(v)$	& =			& $w(uv)$ 			& $\forall \, uv\in M$. 
\end{tabular}
\end{center}

\begin{Proposition}\label{eps_CS_implies_neps_001}
Let $M$ be a perfect matching and $P=(\pi,p)$ dual prices that are $\epsilon$-optimal.  
If $M^*$ is a perfect matching of minimum weight of $G$, then $w(M)\leq w(M^*)+n\epsilon$.
\end{Proposition}
\begin{proof}
Directly from the definition of $\epsilon$-optimal solutions, we get that
\[
w(M)=\sum_{uv\in M}\; w(uv)=\sum_{uv\in M}\;(\pi(u)+p(v))\leq \sum_{uv\in M^*}\; (w(uv)+\epsilon) = w(M^*)+n\epsilon.\vspace{-9mm}
\]
\end{proof}

Since we have integral weights, it follows from Proposition~\ref{eps_CS_implies_neps_001} that if $M$ 
and $P$ are $\epsilon$-optimal solutions for $\epsilon<1/n$, then $M$ is of minimum weight.


\subsection{The auxiliary subgraph $G_{cs}$}\label{G_cs_01}

Given a weighted bipartite graph $\{G,w\}$ and optimal prices $P=(\pi,p)$, we define $G_{cs}(P)$ as follows: 

\begin{Definition}\label{E_min_def_001}
Let $G_{cs}(P)=(U\sqcup V, E_{cs})$, where $E_{cs}(P) = \{ uv\in E \;|\; \pi(u)+p(v)=c(uv) \}$
\end{Definition}

Note that $G_{cs}(P)$ is obtained by removing the edges $uv$ of $G$ such that $w(uv)-\pi(u)-p(v)\neq 0$.
Directly from the definition of $G_{cs}$ we get the following result:

\begin{Proposition}\label{optimum_iif_in_G_min_001}
The subgraph $G_{cs}(P)$ can be constructed in linear $O(m)$ time.
\end{Proposition}

Now, let $\mathcal{M}(G,w)$ the set of all the minimum weight perfect matchings of 
$\{G,w\}$ and $\mathcal{M}(G_{cs}(P))$ the set of all the perfect matchings of $G_{cs}(P)$. 
The following theorem is a key result of this paper, and gives a very important relation between this two sets.

\begin{Theorem}\label{correspondence}
If $P=(\pi,p)$ are optimal prices, then $\mathcal{M}(G,w)=\mathcal{M}(G_{cs}(P))$.
\end{Theorem}
\begin{proof}
Since $P$ are optimal prices, then by corollary~\ref{optimal_face_001}, a perfect matching $M$ of $G$ is of minimum 
weight if and only if $uv\notin M$ for all edge $uv$ such that $w(uv)-\pi(u)-p(v)\neq 0$, which happens if and only if $M$ is in $G_{cs}(P)$.
\end{proof}

Note that Theorem~\ref{correspondence} implies that $\mathcal{M}(G_{cs}(P))$ is the same for all optimal prices $P$.
Regardless of this fact, the following example shows that the subgraph $G_{cs}(P)$ can be different for different optimal prices.

\begin{Example}\label{ejemplo1}
Let $G$ be a bipartite graph with $U=\{u_0, u_1, u_2\}$ and $V=\{v_0, v_1, v_2\}$ as in figure~\ref{G_min_can_be_diferent_001}.
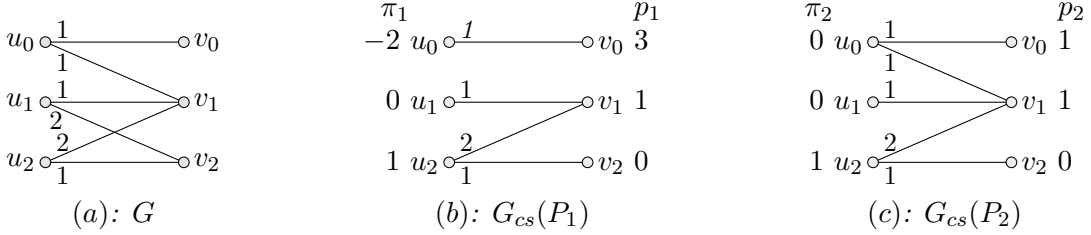
\begin{figure}[ht]
\begin{tabular}{c@{\extracolsep{18mm}}c@{\extracolsep{18mm}}c}
\begin{tikzpicture}[scale=0.8]
\begin{scope}[inner sep=0.5mm, color=black, line width=2pt,fill=black!10]
\draw[thin]	(0,0) -- node[very near start,above] {\small $1$} (2.3,0);
\draw[thin]	(0,0) -- node[very near start,below] {\small $1$} (2.3,-1);
\draw[thin]	(0,-1)	-- node[very near start,above] {\small $1$} (2.3, -1);
\draw[thin]	(0,-1)	-- node {} (2.3, -2);
\draw[thin]	(0,-2)	-- node[very near start,above] {\small $2$}(2.3, -1);
\draw (0.18, -1.3) node {\small $2$};
\draw[thin]	(0,-2)	-- node[very near start,below] {\small $1$} (2.3, -2);
\draw[thin] (0, 0) node [shape=circle,draw, fill, label=left:$u_0$] {};
\draw[thin] (2.3	, 0)	node [shape=circle,draw, fill, label=right:$v_0$] {};
\draw[thin] (0, -1) node [shape=circle,draw, fill, label=left:$u_1$] {};
\draw[thin] (2.3	, -1) node [shape=circle,draw, fill, label=right:$v_1$] {};
\draw[thin] (0, -2) node [shape=circle,draw, fill, label=left:$u_2$] {};
\draw[thin] (2.3	, -2) node [shape=circle,draw, fill, label=right:$v_2$] {};
\end{scope}
\end{tikzpicture}

&

\begin{tikzpicture}[scale=0.8]
\begin{scope}[inner sep=0.5mm, color=black, fill=\colorfill]
\draw[thin]	(0,0)				-- node[very near start,above] {\small 1} (2.3, 0);
\draw[thin]	(0,-1)	-- node[very near start,above] {\small $1$} (2.3, -1);
\draw[thin]	(0,-2)	-- node[very near start,above] {\small $2$} (2.3, -1);
\draw[thin]	(0,-2)	-- node[very near start,below] {\small $1$} (2.3, -2);
\draw (-0.9,0.5) node {$\pi_1$};
\draw (2.3+1, 0.5) node {$p_1$};
\draw[thin] (0, 0)	node [shape=circle,draw, fill, label=left:$-2\, \, u_0$] {};
\draw[thin] (2.3	, 0)	node [shape=circle,draw, fill, label=right:$v_0\, \, 3$] {};
\draw[thin] (0, -1) node [shape=circle,draw, fill, label=left:$0\, \, u_1$] {};
\draw[thin] (2.3	, -1)	node [shape=circle,draw, fill, label=right:$v_1\, \, 1$] {};
\draw[thin] (0, -2) node [shape=circle,draw, fill, label=left:$1\, \, u_2$] {};
\draw[thin] (2.3	, -2)	node [shape=circle,draw, fill, label=right:$v_2\, \, 0$] {};
\end{scope}
\end{tikzpicture}

& 
\begin{tikzpicture}[scale=0.8]
\begin{scope}[inner sep=0.5mm, color=black, fill=\colorfill]
\draw[thin]	(0,0) -- node[very near start,above] {\small $1$} (2.3, 0);
\draw[thin]	(0,0) -- node[very near start,below] {\small $1$} (2.3, -1);
\draw[thin]	(0,-1) -- node[very near start,above] {\small $1$} (2.3, -1);
\draw[thin]	(0,-2) -- node[very near start,above] {\small $2$} (2.3, -1);
\draw[thin]	(0,-2) -- node[very near start,below] {\small $1$} (2.3, -2);
\draw (-0.9,0.5) node {$\pi_2$};
\draw (2.3+1, 0.5) node {$p_2$};
\draw[thin] (0, 0) node [shape=circle,draw, fill, label=left:$0 \, \,  u_0$] {};
\draw[thin] (2.3	, 0)	node [shape=circle,draw, fill, label=right:$v_0\, \,  1$] {};
\draw[thin] (0, -1) node [shape=circle,draw, fill, label=left:$0\, \,  u_1$] {};
\draw[thin] (2.3	, -1) node [shape=circle,draw, fill, label=right:$v_1\, \,  1$] {};
\draw[thin] (0, -2) node [shape=circle,draw, fill, label=left:$1\, \,  u_2$] {};
\draw[thin] (2.3	, -2) node [shape=circle,draw, fill, label=right:$v_2\, \,  0$] {};
\end{scope}
\end{tikzpicture}
\\
$(a)$: $G$ & $(b)$: $G_{cs}(P_1)$ & $(c)$: $G_{cs}(P_2)$ \\
\end{tabular}
\caption{{\bf (a)} An integer weighted bipartite graph $\{G,w\}$.
{\bf (b)} Optimal prices $P_1$ and the subgraph $G_{cs}(P_1)$.
{\bf (c)} Optimal prices $P_2$ and the subgraph $G_{cs}(P_2)$.}
\label{G_min_can_be_diferent_001}
\end{figure}

If $P_1=(\pi_1,p_1)$, where $\pi_1$ and $p_1$ are given by the maps $\{u_0\mapsto -2,\, u_1\mapsto 0,\, u_2\mapsto 1\}$
and $\{v_0\mapsto 3,\, v_1\mapsto 1,\, v_2\mapsto 0\}$ as in figure~\ref{G_min_can_be_diferent_001}(b), 
then it is not difficult to check that 
\[
E_{cs}(P_1)=\{u_0v_0,\, u_1v_1,\, u_2v_1,\, u_2v_2\}.
\]
Also, if $P_2=(\pi_2,p_2)$ is given as in figure~\ref{G_min_can_be_diferent_001}(c), then is not difficult to check that 
\[
E_{cs}(P_2)=\{u_0v_0,\, u_0v_1,\, u_1v_1,\, u_2v_1,\, u_2v_2\}
\]
and therefore $G_{cs}(P_1)\neq G_{cs}(P_2)$.
\end{Example}

It is important to remark that corollary~\ref{optimal_face_001} can be applied to any linear program in standard form
to classify all its optimum primal solutions using one optimal dual solution.
For instance, if we consider the transportation problem on a bipartite graph $T$ with supplies, demands, capacities, and per-unit costs.
Using corollary~\ref{optimal_face_001} and one optimal dual solution, we can construct a subgraph $T_{cs}$ of $T$,
such that the set of all the optimal flows of the original instance is equal to the set of all the feasible flows on the subgraph $T_{cs}$
with the same supplies, demands, and capacities.
In this case there also are algorithms that find optimal dual solutions of the transportation problem. 
But in general, the difficult part is to find optimal dual solutions of a linear program in standard form.


\subsection{Constructing $G_{cs}$ from $\epsilon$-optimal prices}\label{sec5}

Several of the algorithms that solve the assignment problem, especially those based on cost scaling techniques, 
solve it by finding $\epsilon$-optimal solutions for a small enough $\epsilon>0$, 
that guarantees the optimality of the primal solution as Proposition~\ref{eps_CS_implies_neps_001} implies.
See for instance~\cite{auctionA_Bert1, gabo_tarj1, Gold_ken2}.

Since the subgraph $G_{cs}$ is constructed using optimal prices, in this section we provide a procedure
to transform $\epsilon$-optimal solutions $M$ and $P_{\epsilon}=(\pi_{\epsilon},p_{\epsilon})$ with $\epsilon\leq 1/(n+1)$ 
into optimal solutions $M$ and $P=(\pi,p)$ in linear $O(n)$ time.
Note that the matching remains the same since it is already optimum (Proposition \ref{eps_CS_implies_neps_001}).
This algorithm finds a value $t\in \{0,\ldots,n\}$ such that if we define
\[
p(v)=\lfloor p_{\epsilon}(v)+t/(n+1) \rfloor \text{ for all }v\in V \text{ and }
\pi(u)=w(uv)-p(v) \text{ for all } uv\in M,
\]
then $P=(\pi,p)$ is optimal.
As we will prove later, it turns out that $t$ is a good value if 
\[
t\neq \lceil (n+1)(\lceil p_{\epsilon}(v) \rceil - p_{\epsilon}(v)) \rceil \text{ }mod\text{ }(n+1) \text{ for all }v\in V.
\]
Since $|V|=n$ and there are $n+1$ possibilities for $t$, then there exists at least one good value for $t$.
Assuming that $\epsilon\leq 1/(n+1)$, the following algorithm shows the complete procedure.

\begin{Algorithm}
\mbox{}\\
Input: $\epsilon$-optimal solutions $M$ and $P_{\epsilon}=(\pi_{\epsilon}, p_{\epsilon})$.\\
Output: Optimal prices $P=(\pi, p)$.

\renewcommand{\labelenumi}{\fbox{\arabic{enumi}} }
\begin{enumerate}
\item {\bf Procedure} get\_optimal$(M, P_{\epsilon})$
\item $\qquad$good$(j)$=true for all $j\in\{0,\ldots,n\}$;
\item $\qquad$good$(\lceil (n+1)(\lceil p_{\epsilon}(v) \rceil - p_{\epsilon}(v)) \rceil \text{ }mod\text{ } (n+1))$=false for all $v\in V$;
\item $\qquad$get $t$ such that good$(t)$==true;
\item $\qquad$$p(v)=\lfloor p_{\epsilon}(v)+t/(n+1) \rfloor$ for all $v\in V$;
\item $\qquad$$\pi(u)=w(uv)-p(v)$ for all $uv\in M$;
\item $\qquad$return $P=(\pi, p)$;
\item {\bf end}
\end{enumerate}
\end{Algorithm}

It is not difficult to see that all the parts in the procedure get\_optimal runs in $O(n)$ time. 
This gives us an overall $O(n)$ time. 
Now, we prove that the procedure returns indeed optimal prices.

\begin{Proposition}\label{prin2}
Let $M$ and $P_{\epsilon}=(\pi_{\epsilon},p_{\epsilon})$ be $\epsilon$-optimal solutions with $\epsilon\leq 1/(n+1)$, 
then  the procedure get\_optimal will return optimal prices $P=(\pi,p)$.
\end{Proposition}
\begin{proof}
First, we will prove that if $uv\in M$, then $w(uv)-p(v)\leq w(uz)-p(z)$ for all $z\in N(u)$.
If $z=v$, then the inequality is clear.
Thus we can assume that $z\in N(u)\setminus v$.
Since $M$ and $P_{\epsilon}$ are $\epsilon$-optimal, then $w(uv)-p_{\epsilon}(v)\leq w(vz)-p_{\epsilon}(z)+\epsilon$. 
Thus,
\begin{eqnarray*}
w(uv)\!-\!p(v)& =& w(uv)\!-\!\lfloor p_{\epsilon}(v)\!+\!t/(n\!+\!1) \rfloor=w(uv)\!+\!\lceil -p_{\epsilon}(w)\!-\!t/(n\!+\!1) \rceil\\
&=& \lceil w(uv)\!-\!p_{\epsilon}(v)\!-\!t/(n\!+\!1) \rceil \overset{\tiny \epsilon-optimality}{\leq} \lceil w(uz)\!-\!p_{\epsilon}(z)\!+\!\epsilon-t/(n\!+\!1) \rceil\\
&\overset{\tiny \epsilon\leq 1/(n\!+\!1)}{\leq}& \lceil w(uz)\!-\!p_{\epsilon}(z)\!+\!1/(n\!+\!1)\!-\!t/(n\!+\!1) \rceil= w(uz)\!+\!\lceil\!-p_{\epsilon}(z)\!-\!(t-1)/(n\!+\!1)\! \rceil\\
&=&w(uz)\!-\!\lfloor p_{\epsilon}(z)\!+\!(t\!-\!1)/(n\!+\!1) \rfloor.
\end{eqnarray*}

In order to follow with the proof we need the following technical result:

\begin{Claim}\label{s_equal_s_1_001}
Let $r\in \mathbb{R}$, $n\in \mathbb{Z}$ and $t\in \{0,\ldots,n\}$. 
If $t\neq \lceil (n+1)(\lceil r \rceil - r) \rceil \text{ }mod\text{ }(n+1)$, then $\lfloor r+(t-1)/(n+1) \rfloor=\lfloor r+t/(n+1) \rfloor$.
\end{Claim}
\begin{proof}
Let $R=\lceil r \rceil$ and $P_i=r+i/(n+1)$ for all $i\in \{0,\ldots,n+1\}$.
Is not difficult to see that $R\in (P_{i-1}, P_i]$ if and only if $i=\lceil (n+1)(R-r) \rceil \text{ }mod\text{ } (n+1)$ 
(see Figure~\ref{s_equal_s_1_002}).
Thus, if $t\neq\lceil (n+1)(R-r) \rceil \text{ }mod\text{ }(n+1)$, then $\lfloor P_{t-1} \rfloor= \lfloor P_{t} \rfloor=R-1$ 
or $\lfloor P_{t-1} \rfloor= \lfloor P_{t} \rfloor=R$.
That is, $\lfloor P_{t-1} \rfloor=\lfloor P_{t} \rfloor$.
\begin{figure}[ht]
\begin{center}
\begin{tikzpicture}[scale=1]
\begin{scope}[inner sep=1.2mm, color=black, fill=\colorfill]

\draw[step=2,style=help lines]		(0,-\heightii/2) grid (\widthii-4, \heightii/2);
\draw[\thinedge]	(\widthii-1.7, 0)		node {$\mathbb{R}$};

\draw[\thinedge,<->]	(-2,0)	--	(\widthii-2	, 0);

\draw[\thickedge]	(0, -\heightii/2)	--	(0, \heightii/2);
\draw[\thickedge]	(10, -\heightii/2)	--	(10, \heightii/2);
\draw[\thickedge]	(4, -\heightii/2)	--	(4, \heightii/2);
\draw[\thickedge]	(6, -\heightii/2)	--	(6, \heightii/2);

\draw[\thickedge]	(5, -\heightii*0.7)	--	(5, \heightii*0.7);

\draw[\thinedge]	(0, -\heightii)		node {$r$};
\draw[\thinedge]	(4, -\heightii)		node {$r+\frac{t-1}{n+1}$};
\draw[\thinedge]	(6, -\heightii)		node {$r+\frac{t}{n+1}$};
\draw[\thinedge]	(10, -\heightii)	node {$r+1$};
\draw[\thinedge]	(6.6, \heightii)		node[auto] {$R=r+\frac{(n+1)(R - r)}{n+1}$};

\draw[\thickedge,draw=white]	(2.6, 0)	--	(3.3, 0);
\draw[\thinedge]	(3, 0)	node {$\cdots$};
\draw[\thickedge,draw=white]	(6.6, 0)	--	(7.3, 0);
\draw[\thinedge]	(7, 0)	node {$\cdots$};
\end{scope}
\end{tikzpicture}
\end{center}
\caption{The equipartition of the interval $[r,r+1]$ defined by $P_i$.}
\label{s_equal_s_1_002}
\end{figure}
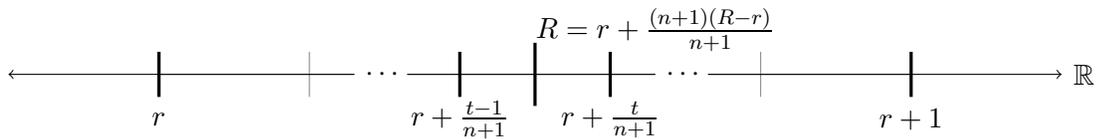
\end{proof}
The authors thank Lyle Ramshaw for having pointed out this identity.

Now, let $t\neq \lceil (n+1)(\lceil p_{\epsilon}(z) \rceil - p_{\epsilon}(z)) \rceil \text{ }mod\text{ }(n+1)$ for all $z\in V$.
Then, by Claim~\ref{s_equal_s_1_001}
\[
w(uv)-p(v)\leq w(uz)-\lfloor p_{\epsilon}(z)+(t-1)/(n+1) \rfloor\overset{\ref{s_equal_s_1_001}}{=} w(uz)-\lfloor p_{\epsilon}(z)+t/(n+1) \rfloor=w(uz)-p(z).
\]
Moreover, if we define $\pi(u) = w(uv)-p(v)$ for each $uv\in M$, then  
$\pi(u)+p(z)\leq w(uz)$ for all $z\in N(u)$, with equality when $z=v$.
Therefore, from Proposition~\ref{complem_slack_004} we get that $M$ and $P$ are optimal solutions.
\end{proof}


\section{Some Applications of the subgraph $G_{cs}$}\label{applications}

In this section we apply the results of section~\ref{FGopt} to solve the weighted variants of some problems involving unweighted perfect matchings in bipartite graphs.
First, we solve the problem of finding all the edges that occur in some minimum weight perfect matching.
Second we solve the problem of enumerating all the minimum weight perfect matchings of a weighted bipartite graph.
The unweighted version of this two problems are addressed in \cite{tassa} and \cite{Takeaki2} respectively.
Finally, we show how we can use the subgraph $G_{cs}$ to solve efficiently an interesting problem which we call the preallocation problem. 
This problem consists of finding a minimum weight perfect matching which contains the maximum number of edges of a prescribed subset of edges.

\subsection{Finding all the edges that occur in some minimum weight perfect matching}
Given a weighted bipartite graph $\{G,w\}$ we want to find the subgraph $G_{opt}=(U\sqcup V, E_{opt})$ on the same vertex than $G$ and edge set given by:
\begin{equation}\label{E_opt_def}
E_{opt}=\{ e\in M\;:\; M \text{ is a minimum weight perfect matching} \}.
\end{equation}

Note that $E_{opt}$ is equal to the union of all the minimum weight perfect matchings.
The unweighted version has been addressed by several authors, see for instance~\cite{tassa, costa, regin}.

Given optimal prices $P=(\pi,p)$, Theorem~\ref{correspondence} gives us a way to compute the subgraph $G_{opt}$. 
First, we construct $G_{cs}(P)$ and then we obtain $E_{opt}$ by finding all the edges that belong to at least one perfect matching of $G_{cs}(P)$. 
This can be done using the algorithm given in~\cite{tassa}.
The following proposition summarizes one of the main results of~\cite{tassa}. 

\begin{Proposition}{\cite{tassa}}\label{tassa_prop}
Given a feasible bipartite graph $G=(U\sqcup V, E)$. 
All the edges that occur in some perfect matching of $G$ can be found in $O(m)$ time.
\end{Proposition}

Assuming that the bipartite graph is feasible, the following algorithm finds its subgraph $G_{opt}$.
\begin{Algorithm}\label{g_opt_alg}
\mbox{}\\
Input: An integer weighted bipartite graph $\{G,w\}$ and optimal dual prices $P$.\\
Output: The subgraph $G_{opt}$ of $G$.

\renewcommand{\labelenumi}{\fbox{\arabic{enumi}} }
\begin{enumerate}
\item {\bf Procedure} get\_Gopt$(G,w,P)$
\item $\qquad$ Construct the subgraph $G_{cs}(P)=(U\sqcup V, E_{cs})$;
\item $\qquad$ Using  \cite[Algorithm 2]{tassa} get $E_{opt}=\{e\in M \, | \, M \text{ is a perfect matching of } G_{cs}(P) \}$;
\item $\qquad$ Return $G_{opt}=(U\sqcup V, E_{opt})$;
\item {\bf end}
\end{enumerate}
\end{Algorithm}
\begin{Proposition}
Given a weighted bipartite graph $\{G,w\}$, the procedure \textit{get\_Gopt} returns $G_{opt}$. 
\end{Proposition}
\begin{proof}
Note that that the edge set returned by the procedure \textit{get\_Gopt} is equal to $\bigcup\; \{M\;:\; M\in \mathcal{M}(G_{cs}(P)) \}$.
But, by Theorem~\ref{correspondence}, $\mathcal{M}(G_{cs}(P))=\mathcal{M}(G,w)$. 
Then procedure \textit{get\_Gopt} returns the edge set $\bigcup\; \{M\;:\; M\in \mathcal{M}(G,w) \}$, which is equivalent to $E_{opt}$.
\end{proof}

\begin{Theorem}
Given a weighted bipartite graph $\{G,w\}$ and optimal dual prices $P$, all the edges of $G$ that occur in some
minimum weight perfect matching can be found in $O(m)$ time.
\end{Theorem}
\begin{proof}
By Propositions~\ref{optimum_iif_in_G_min_001} and~\ref{tassa_prop}, $G_{cs}(P)$ can be obtained 
in $O(m)$ time and $E_{opt}$ can be constructed in $O(|E_{cs}(P)|)$ time.
Therefore algorithm \ref{g_opt_alg} works on $O(m)$ overall  time.
\end{proof}


\subsection{Enumerating all the minimum weight perfect matchings}\label{enumerating}

In this section we give an algorithm to enumerate all the minimum weight perfect matchings of $\{G,w\}$.
Given $G_{cs}(P)$ for some optimal dual prices $P$ we can enumerate all the perfect 
matchings of $G_{cs}(P)$ using some of the algorithms contained in~\cite{Takeaki2,FukudaMatsui}.
For instance, in~\cite{Takeaki2} we find the following result:
\begin{Theorem}{\cite[Theorem 1]{Takeaki2}.}\label{takeaki_theorem}
Perfect matchings in a bipartite graph $G=(V, E)$ can be enumerated in $O(|E|\sqrt{|V|})$ preprocessing time and
$O(\log |V|)$ time per perfect matching.
\end{Theorem}

Assuming that the bipartite graph is feasible, the following algorithm 
describes how to enumerate all the minimum weight perfect matchings.
\begin{Algorithm}\mbox{}\\\label{enum_mwpm_alg}
Input: An integer weighted bipartite graph $\{G,w\}$ and optimal dual prices $P$.\\
Output: The enumeration of all the minimum weight perfect matchings of $\{G,w\}$.

\renewcommand{\labelenumi}{\fbox{\arabic{enumi}} }
\begin{enumerate}
	\item {\bf Procedure} enumerate\_MW\_Per\_Mat$(G,w,P)$
	\item $\qquad$Construct the subgraph $G_{cs}(P)=(U\sqcup V, E_{cs})$;
	\item $\qquad$Using \cite[Algorithm in page 369]{Takeaki2} enumerate all the perfect matchings of $G_{cs}(P)$;
	\item $\qquad$return;
	\item {\bf end}
\end{enumerate}
\end{Algorithm}
\begin{Proposition}
The procedure enumerate\_MW\_Per\_Mat enumerates all the minimum weight perfect matchings of a weighted bipartite graph $\{G,w\}$.
\end{Proposition}
\begin{proof}
It follows because by Theorem~\ref{correspondence}, $\mathcal{M}(G_{cs}(P))=\mathcal{M}(G,w)$.
\end{proof}

\begin{Theorem}
Given a weighted bipartite graph $\{G,w\}$ and optimal dual prices $P$, all its minimum weight perfect matchings
can be enumerated in $O(m+|E_{cs}(P)|\sqrt{n}+|\mathcal{M}(G,w)|\log\,n)$ time.
\end{Theorem}
\begin{proof}
By Proposition~\ref{optimum_iif_in_G_min_001}, the subgraph $G_{cs}(P)$ can be obtained in $O(m)$ time.
Also, by Theorem~\ref{takeaki_theorem}, we can enumerate all the perfect matchings of $G_{cs}(P)$ in $O(|E_{cs}(P)|\sqrt{n}+\mathcal{M}(G,w)\log\,n)$ time.
Therefore, Algorithm~\ref{enum_mwpm_alg} runs in $O(m+|E_{cs}(P)|\sqrt{n}+|\mathcal{M}(G,w)|\log\,n)$ time.
\end{proof}

Note that $G_{cs}(P)$ can have a lot less edges than $G$ and therefore
$G$ can have a very small number of minimum weight perfect matchings in such a way that 
$O(m+|E_{cs}(P)|\sqrt{n}+\mathcal{M}(G,w)\log\,n)=O(m)$.


\subsection{The preallocation problem}\label{preallocation}

The {\it preallocation problem} can be stated as follows:
Given a weighted bipartite graph $\{G,w\}$ and a subset of edges $E_p\subseteq E(G)$,
we want to find a minimum weight perfect matching $M$ of $\{G,w\}$ such that $|M\cap E_p|$ is maximum.
In other words, there is no other minimum weight perfect matching of $\{G,w\}$ that has more edges of $E_p$ than $M$.
We can see the set $E_p$ as a set of preferences.
Note that there can be several minimum weight perfect matchings that respect a maximum number of preferences.

This problem can be easily solved with the help of the subgraph $G_{cs}$.
Given the subgraph $G_{cs}(P)$ for some optimal dual prices $P$, we define a weight function $w_p:E(G_{cs}(P))\rightarrow \{0,1\}$
over the edges of $G_{cs}(P)$, as follows:
\[w_p(e)=
\begin{cases}
0 & \text{ if } e\in E_p, \\
1 & \text{ if } e\notin E_p.
\end{cases}
\]
The algorithm can be derived from the following proposition.
\begin{Proposition}\label{pre_prop_01}
If $M$ is a minimum weight perfect matching of $\{G_{cs}(P), w_p\}$, then $M$ is a minimum weight perfect matching
of $\{G,w\}$ such that $|M\cap E_p|$ is maximum. 
\end{Proposition}
\begin{proof}
Since $P$ are optimal dual prices, then from Theorem~\ref{correspondence} follows that $M$ is a minimum weight perfect matching of $\{G,w\}$. 
Also, since $\mathcal{M}(G_{cs}(P))=\mathcal{M}(G,w)$, then $M$ maximizes $|M\cap E_p|$ over the elements of $\mathcal{M}(G,w)$
if and only if it maximizes $|M\cap E_p|$ over the elements of $\mathcal{M}(G_{cs}(P))$.

Now, assume that $M'$ is a perfect matching of $G_{cs}(P)$ containing more edges of $E_p$ than $M$.
Since $M'$ has more edges with zero weight than $M$, then $w(M')< w(M)$ on $\{G_{cs}(P), w_p\}$,
a contradiction to the optimality of $M$. 
Therefore $M$ maximizes $|M\cap E_p|$ over the elements of $\mathcal{M}(G_{cs}(P))$.
\end{proof}

Assuming that the bipartite graph is feasible, the following algorithm shows how to solve the prea-llocation problem.
\begin{Algorithm}\mbox{}\\ \label{preall_alg}
Input: An integer weighted bipartite graph $\{G,w\}$, optimal dual prices $P$ and a subset of edges $E_p$.\\
Output: A minimum weight perfect matching that contains a maximum number of edges of $E_p$.

\renewcommand{\labelenumi}{\fbox{\arabic{enumi}} }
\begin{enumerate}
	\item {\bf Procedure} preallocation$(G,w,P,E_p)$
	\item $\qquad$Construct the subgraph $G_{cs}(P)=(U\sqcup V, E_{cs})$;
	\item $\qquad$Construct the weight function $w_p$, as given above;
	\item $\qquad$Get a minimum weight perfect matching $M$ of $\{G_{cs}(P), w_p\}$;
	\item $\qquad$return $M$;
	\item {\bf end}
\end{enumerate}
\end{Algorithm}

\begin{Proposition}
Algorithm~\ref{preall_alg} returns an optimum matching wich maximizes $|M\cap E_p|$.
\end{Proposition}
\begin{proof}
It follows directly from Proposition~\ref{pre_prop_01}.
\end{proof}

\begin{Theorem}
Given a weighted bipartite graph $\{G,w\}$, optimal dual prices $P$ and a subset of edges $E_p$. 
A minimum weight perfect matching that maximizes $|M\cap E_p|$, can be found in $O(\sqrt{n}\,m\log\,n)$ time.
\end{Theorem}
\begin{proof}
By Proposition~\ref{optimum_iif_in_G_min_001} the subgraph $G_{cs}(P)$ can be obtained in $O(m)$ time.
Also, the weight function $w_p$ can be constructed in $O(m)$. 
Finally, by the algorithm given in~\cite[Section 2.1]{gabo_tarj1} we have that the minimum weight perfect matching 
of $\{G_{cs}(P),w_p\}$ can be obtained in $O(\sqrt{n}\,m\log\,n)$ time.
Giving a $O(\sqrt{n}\,m\log\,n)$ total time.
\end{proof}


\section{Strategies for unfeasible bipartite graphs}\label{FGoptNPM}

In this section we deal with the case when $G$ does not have necessarily a perfect matching,
and we are interested in minimum weight maximum cardinality matchings.
A simple way to compute $G_{cs}$ in this case, is by constructing a feasible weighted bipartite graph $\{G',w'\}$, such that a
minimum weight perfect matching of $\{G',w'\}$ induces a minimum weight maximum cardinality matching of $\{G,w\}$ and vice versa. 
Then we solve the problem that we are addressing in $\{G',w'\}$ and translate the solution obtained to a solution on the original weighted bipartite graph $\{G,w\}$.

We will give three graph transformation that satisfies the above conditions.
Let $\{G,w\}$ be an integer weighted bipartite graph with $U=\{ u_1,\ldots,u_n \}$, $V=\{ v_1,\ldots,v_s \}$, $n\geq s$, $m=|E|$, and $W=\max_{e\in E}\;\{|w(e)|\}$.


\begin{figure}[ht]
\begin{tabular}{c@{\extracolsep{10mm}}c@{\extracolsep{10mm}}c}
\begin{tikzpicture}[scale=0.45]
\begin{scope}[thin, inner sep=0.4mm, color=black, line width=2pt,fill=black!10]
\begin{scope}[color=blue,dashed]
\draw[thin] (0,0) -- node[above] {} (6,-4);
\draw[thin] (0,-1) -- node[above] {} (6,-5);
\draw[thin] (0,-2) -- node[above] {} (6,-6);
\draw[thin] (0,-3) -- node[above] {} (6,-7);
\draw[thin] (0,-4.4) -- node[above] {} (6,-8.4);
\draw[thin] (6,0) -- node[above] {} (0,-6);
\draw[thin] (6,-1) -- node[above] {} (0,-7);
\draw[thin] (6,-3) -- node[above] {} (0,-8.4);
\end{scope}
\draw[thin]	(0,0) -- node[very near start,above] {\small $3$} (6,0);
\draw[thin]	(0,0) -- node[near start,below] {\small $1$} (6,-1);
\draw[thin]	(0,-1) -- node[very near start,above] {\small $7$} (6,-3);
\draw[thin]	(0,-2) -- node[very near start,below] {\small $8$} (6,-1);
\draw[thin]	(0,-3) -- node[near start,above] {\small $2$} (6,-3);
\draw[thin]	(0,-4.4) -- node[very near start,above] {\small $1$} (6,-1);
\draw[thin]	(6,-4) -- node[very near start,above] {\small $3$} (0,-6);
\draw[thin]	(6,-4) -- node[near start,below] {\small $1$} (0,-7);
\draw[thin]	(6,-5) -- node[very near start,above] {\small $7$} (0,-8.4);
\draw[thin]	(6,-6) -- node[near start,below] {\small $8$ } (0,-7);			
\draw[thin]	(6,-7) -- node[very near start,below] {\small $2$} (0,-8.4);
\draw[thin]	(6,-8.4) -- node[very near start,below] {\small $1$} (0,-7);
\begin{scope}[inner sep=0.4mm]
\draw[thin] (0,0) node [shape=circle,draw,fill=black!20,label=left:$u_1$] {};
\draw[thin] (6,0) node [shape=circle,draw,fill=black!20,label=right:$v_1$] {};
\draw[thin] (0,-1) node [shape=circle,draw,fill=black!20,label=left:$u_2$] {};
\draw[thin] (6,-1) node [shape=circle,draw,fill=black!20,label=right:$v_2$] {};
\draw[thin] (0,-2) node [shape=circle,draw,fill=black!20,label=left:$u_3$] {};
\draw[thin] (6,-2) node {$\vdots$};
\draw[thin] (0,-3) node [shape=circle,draw,fill=black!20,label=left:$u_4$] {};
\draw[thin] (6,-3) node [shape=circle,draw,fill=black!20,label=right:$v_s$] {};
\draw[thin] (0,-3.5) node {$\vdots$};
\draw[thin] (0,-4.4) node [shape=circle,draw,fill=black!20,label=left:$u_n$] {};
\end{scope}
\begin{scope}[inner sep=0.4mm]
\draw[thin] (6,-4) node [shape=circle,draw,fill=black!20,label=right:$u'_1$] {};
\draw[thin] (0,-6) node [shape=circle,draw,fill=black!20,label=left:$v'_1$] {};
\draw[thin] (6,-5) node [shape=circle,draw,fill=black!20,label=right:$u'_2$] {};
\draw[thin] (0,-7) node [shape=circle,draw,fill=black!20, label=left:$v'_2$] {};
\draw[thin] (6,-6) node [shape=circle,draw,fill=black!20,label=right:$u'_3$] {};
\draw[thin] (0,-7.5) node {$\vdots$};
\draw[thin] (6,-7) node [shape=circle,draw,fill=black!20,label=right:$u'_4$] {};
\draw[thin] (0,-8.4) node [shape=circle,draw,fill=black!20,label=left:$v'_s$] {};
\draw[thin] (6,-7.5) node {$\vdots$};
\draw[thin] (6,-8.4) node [shape=circle,draw,fill=black!20,label=right:$u'_n$] {};
\end{scope}
\draw[thin] (-1.1, 0.4) .. controls (-1.6, 0.4) and (-1.1, -2.5) .. (-1.6, -2.5);
\draw[thin] (-1.6, -2.5) .. controls (-1.1, -2.5) and (-1.6, -5.4) .. (-1.1, -5.4);
\draw[thin] (-2.1, -2.5) node {$U$};
\draw[thin] (-1.1, -5.6) .. controls (-1.6, -5.6) and (-1.1, -7.3) .. (-1.6, -7.3);
\draw[thin] (-1.6, -7.3) .. controls (-1.1, -7.3) and (-1.6, -9) .. (-1.1, -9);
\draw[thin] (-2.2, -7.3) node {$V'$};
\draw[thin] (7.1, 0.4) .. controls (7.6, 0.4) and (7.1, -1.5) .. (7.6, -1.5);
\draw[thin] (7.6, -1.5) .. controls (7.1, -1.5) and (7.6, -3.4) .. (7.1, -3.4);
\draw[thin] (8.1, -1.5) node {$V$};
\draw[thin] (7.1, -3.6) .. controls (7.6, -3.6) and (7.1, -6.3) .. (7.6, -6.3);
\draw[thin] (7.6, -6.3) .. controls (7.1, -6.3) and (7.6, -9) .. (7.1, -9);
\draw[thin] (8.3, -6.3) node {$U'$};			
\end{scope}
\end{tikzpicture}
&
\begin{tikzpicture}[scale=0.45]
\begin{scope}[thin, inner sep=0.4mm, color=black, line width=2pt,fill=black!10]
\begin{scope}[color=blue,dashed]
\draw[thin] (0,0) -- node[above] {} (6,-4);
\draw[thin] (0,-1) -- node[above] {} (6,-5);
\draw[thin] (0,-2) -- node[above] {} (6,-6);
\draw[thin] (0,-3) -- node[above] {} (6,-7);
\draw[thin] (0,-4.4) -- node[above] {} (6,-8.4);
\end{scope}

\draw[thin]	(0,0) -- node[very near start,above] {\small $3$} (6,0);
\draw[thin]	(0,0) -- node[near start,below] {\small $1$} (6,-1);
\draw[thin]	(0,-1) -- node[very near start,above] {\small $7$} (6,-3);
\draw[thin]	(0,-2) -- node[very near start,below] {\small $8$} (6,-1);
\draw[thin]	(0,-3) -- node[near start,above] {\small $2$} (6,-3);
\draw[thin]	(0,-4.4) -- node[very near start,above] {\small $1$} (6,-1);
\draw[thin]	(6,-4) -- node[very near start,above] {\small $3$} (0,-6);
\draw[thin]	(6,-4) -- node[near start,below] {\small $1$} (0,-7);
\draw[thin]	(6,-5) -- node[very near start,above] {\small $7$} (0,-8.4);
\draw[thin]	(6,-6) -- node[near start,below] {\small $8$ } (0,-7);			
\draw[thin]	(6,-7) -- node[very near start,below] {\small $2$} (0,-8.4);
\draw[thin]	(6,-8.4) -- node[very near start,below] {\small $1$} (0,-7);
\begin{scope}[inner sep=0.4mm]
\draw[thin] (0,0) node [shape=circle,draw,fill=black!20,label=left:$u_1$] {};
\draw[thin] (6,0) node [shape=circle,draw,fill=black!20,label=right:$v_1$] {};
\draw[thin] (0,-1) node [shape=circle,draw,fill=black!20,label=left:$u_2$] {};
\draw[thin] (6,-1) node [shape=circle,draw,fill=black!20,label=right:$v_2$] {};
\draw[thin] (0,-2) node [shape=circle,draw,fill=black!20,label=left:$u_3$] {};
\draw[thin] (6,-2) node {$\vdots$};
\draw[thin] (0,-3) node [shape=circle,draw,fill=black!20,label=left:$u_4$] {};
\draw[thin] (6,-3) node [shape=circle,draw,fill=black!20,label=right:$v_s$] {};
\draw[thin] (0,-3.5) node {$\vdots$};
\draw[thin] (0,-4.4) node [shape=circle,draw,fill=black!20,label=left:$u_n$] {};
\end{scope}
\begin{scope}[inner sep=0.4mm]
\draw[thin] (6,-4) node [shape=circle,draw,fill=black!20,label=right:$u'_1$] {};
\draw[thin] (0,-6) node [shape=circle,draw,fill=black!20,label=left:$v'_1$] {};
\draw[thin] (6,-5) node [shape=circle,draw,fill=black!20,label=right:$u'_2$] {};
\draw[thin] (0,-7) node [shape=circle,draw,fill=black!20, label=left:$v'_2$] {};
\draw[thin] (6,-6) node [shape=circle,draw,fill=black!20,label=right:$u'_3$] {};
\draw[thin] (0,-7.5) node {$\vdots$};
\draw[thin] (6,-7) node [shape=circle,draw,fill=black!20,label=right:$u'_4$] {};
\draw[thin] (0,-8.4) node [shape=circle,draw,fill=black!20,label=left:$v'_s$] {};
\draw[thin] (6,-7.5) node {$\vdots$};
\draw[thin] (6,-8.4) node [shape=circle,draw,fill=black!20,label=right:$u'_n$] {};
\end{scope}
\draw[thin] (-1.1, 0.4) .. controls (-1.6, 0.4) and (-1.1, -2.5) .. (-1.6, -2.5);
\draw[thin] (-1.6, -2.5) .. controls (-1.1, -2.5) and (-1.6, -5.4) .. (-1.1, -5.4);
\draw[thin] (-2.1, -2.5) node {$U$};
\draw[thin] (-1.1, -5.6) .. controls (-1.6, -5.6) and (-1.1, -7.3) .. (-1.6, -7.3);
\draw[thin] (-1.6, -7.3) .. controls (-1.1, -7.3) and (-1.6, -9) .. (-1.1, -9);
\draw[thin] (-2.2, -7.3) node {$V'$};
\draw[thin] (7.1, 0.4) .. controls (7.6, 0.4) and (7.1, -1.5) .. (7.6, -1.5);
\draw[thin] (7.6, -1.5) .. controls (7.1, -1.5) and (7.6, -3.4) .. (7.1, -3.4);
\draw[thin] (8.1, -1.5) node {$V$};
\draw[thin] (7.1, -3.6) .. controls (7.6, -3.6) and (7.1, -6.3) .. (7.6, -6.3);
\draw[thin] (7.6, -6.3) .. controls (7.1, -6.3) and (7.6, -9) .. (7.1, -9);
\draw[thin] (8.3, -6.3) node {$U'$};			
\end{scope}
\end{tikzpicture}
&
\begin{tikzpicture}[scale=0.45]
\begin{scope}[thin, inner sep=0.4mm, color=black, line width=2pt,fill=black!10]
\draw[thin]	(0,0) -- node[very near start,above] {\small $3$} (6,0);
\draw[thin]	(0,0) -- node[near start,below] {\small $1$} (6,-1);
\draw[thin]	(0,-1) -- node[very near start,above] {\small $7$} (6,-3);
\draw[thin]	(0,-2) -- node[very near start,above] {\small $8$} (6,-1);
\draw[thin]	(0,-3) -- node[very near start,above] {\small $2$} (6,-3);
\draw[thin]	(0,-5) -- node[very near start,above] {\small $1$} (6,-1);
\begin{scope}[color=blue,dashed]
\draw[thin] (0,0) -- node[above] {} (6,-4);
\draw[thin] (0,-1) -- node[above] {} (6,-4);
\draw[thin] (0,-2) -- node[above] {} (6,-4);
\draw[thin] (0,-3) -- node[above] {} (6,-4);
\draw[thin] (0,-4) -- node[above] {} (6,-4);
\draw[thin] (0,-5) -- node[above] {} (6,-4);
\draw[thin] (0,0) -- node[above] {} (6,-5);
\draw[thin] (0,-1) -- node[above] {} (6,-5);
\draw[thin] (0,-2) -- node[above] {} (6,-5);
\draw[thin] (0,-3) -- node[above] {} (6,-5);
\draw[thin] (0,-4) -- node[above] {} (6,-5);
\draw[thin] (0,-5) -- node[above] {} (6,-5);
\end{scope}
\begin{scope}[inner sep=0.4mm]
\draw[thin] (0,0) node [shape=circle,draw,fill=black!20, label=left:$u_1$] {};
\draw[thin] (6,0) node [shape=circle,draw,fill=black!20,label=right:$v_1$] {};
\draw[thin] (0,-1) node [shape=circle,draw,fill=black!20,label=left:$u_2$] {};
\draw[thin] (6,-1) node [shape=circle,draw,fill=black!20,label=right:$v_2$] {};
\draw[thin] (0,-2) node [shape=circle,draw,fill=black!20,label=left:$u_3$] {};
\draw[thin] (6,-2) node {$\vdots$};
\draw[thin] (0,-3) node [shape=circle,draw,fill=black!20,label=left:$u_4$] {};
\draw[thin] (6,-4) node {$\vdots$};
\draw[thin] (6,-5) node [shape=circle,draw,fill=black!20,label=right:$v_n$] {};
\draw[thin] (6,-3) node [shape=circle,draw,fill=black!20,label=right:$v_s$] {};
\draw[thin] (0,-4) node {$\vdots$};
\draw[thin] (0,-5) node [shape=circle,draw,fill=black!20,label=left:$u_n$] {};
\end{scope}
\draw[thin] (0, -6) node { $U$};
\draw[thin] (6, -6) node { $V_a$};
\end{scope}
\end{tikzpicture}\\
(a) $\{G_d,w_d\}$ & (b) $\{G_s,w_s\}$ & (c) $\{G_a,w_a\}$
\end{tabular}
\caption{(a) The first doubling transformation, (b) the second doubling transformation, and (c) the artificial vertices transformation.
}
\label{to_doubling_transfromation_0}
\end{figure}
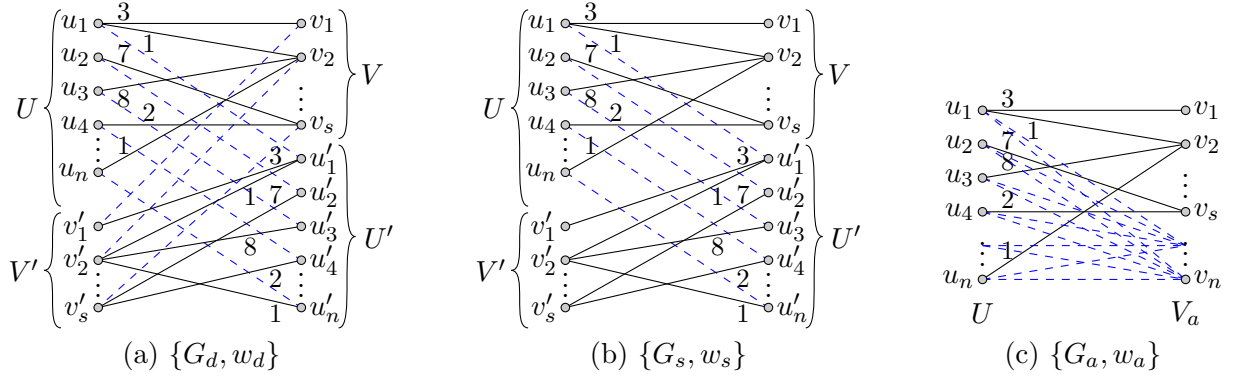


\subsection{The first doubling transformation}
Let $G_d=(U_d \sqcup V_d, E_d)$ with $U_d=U\sqcup V'$, $V_d=V\sqcup U'$, and
$E_d=E\cup E_f\cup E_U\cup E_V$, where $U'=\{ u'_1,\ldots,u'_n \}$ is a copy of $U$, $V'=\{v'_1,\ldots,v'_s\}$ is a copy of $V$, 
$E_f=\{ v'_ju'_i  \;|\; u_iv_j\in E \}$ is a flipped copy of $E$, $E_U=\{ u_iu'_i \;|\; u_i\in U \}$, and $E_V=\{ v'_jv_j \;|\; v_j\in V \}$, see figure~\ref{to_doubling_transfromation_0}. 
Also, let $w_d:E' \rightarrow \mathbb{Z}$ be given by:
\[
w_d(e)=
\begin{cases}
w(e) & \text{ for all } e\in E,\\
w(u_iv_j) & \text{ for all } e=v'_ju'_i\in E_f,\\
2sW & \text{ for all } e\in E_U\cup E_V.
\end{cases}
\]

Note that $E_U\cup E_V$ is a perfect matching of $G_d$. 
The following result gives us a correspondence between the optimum matchings of $\{G_d,w_d\}$ and $\{G,w\}$.

\begin{Proposition}\label{doubling_trans_theorem_0}
Let $\{G,w\}$ be an integer weighted bipartite graph.
If $N$ is a minimum weight perfect matching of $\{G_d,w_d\}$, then $M=N\cap E$ is an optimum matching of $\{G,w\}$.
\end{Proposition}
\begin{proof}
Given a matching $A$ of $G$, let $flip(A)=\{u'v' \, | \, uv\in A\}$.
If $M_f=N\cap E_f$, then $M$ and $M_f$ have the same cardinality because $M$ and $flip(M_f)$ cover the same vertices of G.
Now, assume that $M$ is not an optimum matching, that is, there exist a matching $M^*$ such that $|M|< |M^*|$ or $w(M) > w(M^*)$. 
Let $N'=M^*\cup flip(M^*)\cup I$, where $I=\{zz'\, | \, z\in V(G) \text{ is not cover by } M^*\}$. 
Since we can assume without loss of generalization that $w(M)={\rm min}(w(M),w(M_f))$, then
\[
w(N)\geq (n+s-2|M|)2sW+2w(M)> (n+s-2|M^*|)2sW+2w(M^*)=w(N');
\]
a contradiction to the fact that $N$ is of minimum weight. 
\end{proof}

Let us denote by MWPM the Minimum Weight Perfect Matching problem and by OM the Optimum Matching problem.
Now, we will show that both problems are reducible.

\begin{Proposition}\label{MWPM_MWMM_01}
The MWPM and OM problems are mutually reducible.
\end{Proposition}

\begin{proof}
First, the reduction from MWPM to OM is trivial, because a perfect matching is of maximum cardinality. 
On the other hand, the reduction from OM to MWPM is obtained by Proposition~\ref{doubling_trans_theorem_0}.
\end{proof}

The graph $G_d$ has the disadvantage that the weights of the edges $E_U\cup E_V$ increase with $s$.
Therefore it can produce overflow errors for the computer data types on relatively small graphs.
In order to avoid this disadvantage we present other two transformations, which also give us a reduction between OM and MWPM.
However, this two transformations need to assume that $G$ has a matching covering the small side $V$. 


\subsection{The second doubling transformation}
To the knowledge of the authors, this is a new variant of the previous transformation.
Let $G_s=G_d\setminus E_V$ and $w_s:E' \rightarrow \mathbb{Z}$ be given by: 
\[
w_s(e)=
\begin{cases}
w(e) & \text{ for all } e\in E \\
w(u_iv_j) & \text{ for all } e=v'_ju'_i\in E_f \\
k & \text{ for all } e\in E_U,
\end{cases}
\]
where $k\in \mathbb{Z}$ is any constant.

In this case we cannot guarantee that $G_s$ has a perfect matching. 
However, $G_s$ has a perfect matching if and only if $G$ has a matching that covers all the vertices of the small side $V$.
The following result gives us a correspondence between the optimum matchings of $\{G_s,w_s\}$ and $\{G,w\}$.

\begin{Proposition}\label{sdt_prop}
Let $\{G,w\}$ be an integer weighted bipartite graph.
If $N$ is a minimum weight perfect matching of $\{G_s,w_s\}$, then $M=N\cap E$ is an optimum matching of $\{G,w\}$ which covers $V$.
\end{Proposition}
\begin{proof}
First, since $N$ is perfect, $M$ must cover $V$.
The optimality of $M$ follows from the fact that the cost of any perfect matching 
$N$ of $G_s$ is $w(N)=2w(N\cap E)+(n-s)k$ and $(n-s)k$ is constant.
\end{proof}

This transformation has the advantage that the maximum weight of the edges is not increased.
However, when $|U|-|V|$ is small enough ($G$ is almost balanced), we need to solve an instance $G_s$ with the double of edges.
In order to avoid this disadvantage we present a last transformation.


\subsection{The artificial vertices transformation}
The last construction consists in balancing $G$ by adding $|U|-|V|$ ``artificial" vertices to $V$.
More precisely, let $G_a=(U \sqcup V_a, E_a)$ where $V_a=V\cup \{v_{s+1},\ldots,v_n \}$, 
$D=\{ u_iv_j \;|\; u_i\in U \text{ and } v_j\in V_a\setminus V \}$,  and $E_a=E \cup D$; see  figure~\ref{to_doubling_transfromation_0}.
Also, let $w_a:E' \rightarrow \mathbb{Z}$ given by
\[
w_a(e)=
\begin{cases}
w(u_iv_j) & \text{ if } e\in E,\\
k  & \text{ if } e\in D,
\end{cases}
\]
where $k\in \mathbb{Z}$ is any constant.

Clearly $G_a$ has a perfect matching if and only if $G$ has a matching that covers $V$.

\begin{Proposition}
Let $\{G,w\}$ be an integer weighted bipartite graph.
If $N$ is a minimum weight perfect matching of $\{G_a,w_a\}$, then $M=N\cap E$ is an optimum matching of $\{G,w\}$ which covers $V$.
\end{Proposition}
\begin{proof}
Follows by similar arguments of those in Proposition~\ref{sdt_prop}.
\end{proof}

This transformation has the advantage that the number of edges of $G_a$ is only increased a little when $|U|-|V|$ is sufficiently small.
In the counterpart, the density of $G_a$ is increased considerably when $|U|-|V|$ is big.

We can use any of these transformations to solve the problem of finding all the edges that 
occur in any optimum maximum cardinality matching of a weighted bipartite graph $\{G,w\}$. 
We only need to use a transformation $\{G',w'\}$ and find all the edges $E'_{opt}$ that occur in any minimum weight perfect matching of the transformation. 
And the solution to the original instance will be given by $E_{opt}=E'_{opt}\cap E$.
Something similar works for the preallocation problem where we want to maximize the preferences over the optimum maximum cardinality matchings.
However, this does not work for efficiently solve the problem of enumerating all the optimum maximum cardinality matchings. 
Because the same matching can be enumerated a non constant number of times. For instance, consider a complete bipartite graph
with weights $w(e)=1$ for all $e\in E$. This is due to the fact that the relation between the optimum matchings of these 
transformations and the original instance is not a one-to-one relation.


\begin{thebibliography}{100}
\bibitem{burkard}{R. Burkard, M. Dell'Amico, and S. Martello, \textit{Assignment problems}, Society for Industrial and Applied Mathematics (SIAM), Philadelphia, PA, 2009.}

\bibitem{lovasz}{L. Lov\'asz and M.D. Plummer, \textit{Matching theory}, AMS Chelsea Publishing, Providence, RI, 2009.}

\bibitem{line_opt1}{D. Bertsimas and J.N. Tsitsiklis, \textit{Introduction to linear optimization}, Athena Scientific, Belmont, MA, 1997.}

\bibitem{sch}{A. Schrijver, \textit{Combinatorial optimization: Polyhedra and efficiency}, Algorithms and Combinatorics 24, Springer-Verlag, Berlin, 2003.}


\bibitem{Kuhn}{H.W. Kuhn, \textit{The hungarian method for the assignment problem}, Naval Research Logistics Quart. 2, 83--97, 1955.}

\bibitem{Gabow}{H.N. Gabow, \textit{Scaling algorithms for network problems}, Journal of computer \& system sciences 31, 148--168, 1985.}

\bibitem{DinicKronrod}{E.A. Dinic and M.A. Kronrod, \textit{An algorithm for the solution of the assignment problem}, Soviet Math. Dokl. Vol 10, 1324--1326, 1969.}


\bibitem{auctionA_Bert1}{D.P. Bertsekas, \textit{Auction algorithms for network flow problems: A tutorial introduction}, Computational optimization and applications 1, 7--66, 1992.}

\bibitem{gabo_tarj1}{H.N. Gabow and  R.E. Tarjan, \textit{Faster scaling algorithms for network problems}, SIAM J. Computation Vol. 18, No. 5, 1013--1036, 1989.}

\bibitem{Gold_ken2}{A.V. Goldberg and R. Kennedy, \textit{An efficient cost scaling algorithm for the assignment problem}, Math. Programming 71, 153--177, 1995.}


\bibitem{hopcroft}{J. Hopcroft and R.M. Karp, \textit{An $n^{5/2}$ algorithm for maximum matchings in bipartite graphs}, SIAM J. Comput. 2, 225--231, 1973.}

\bibitem{AltBlumMehlhornPaul}{H. Halt, N. Blum, K. Mehlhorn and M. Paul, \textit{Computing a maximum cardinality matching in a bipartite graph in time $O(n^{1.5}\sqrt{m/log\,n})$}, 
Information Processing Letters 37, 237--240, 1991.}

\bibitem{Takeaki2}{T. Uno, \textit{A Fast Algorithm for Enumerating Bipartite Perfect Matchings}, Lecture Notes in Computer Science, Springer Verlag, Vol. 2223, 367--379, 2001.}

\bibitem{FukudaMatsui}{K. Fukuda and T. Matsui, \textit{Finding All the Perfect Matchings in Bipartite Graphs}, Appl. Math. Lett, Vol. 7, No. 1, 15--18, 1994.}


\bibitem{tassa}{T. Tassa, \textit{Finding All Maximally-Matchable Edges in a Bipartite Graph}, Theoretical Computer Science 423, 50--58, 2012.}

\bibitem{costa}{M.C. Costa, \textit{Persistency in maximum cardinality bipartite matchings}, Oper. Res. Lett. 15,143--149, 1994.}


\bibitem{regin}{J.C. R\'egin, \textit{A filtering algorithm for constraints of difference in CSPs}, Proceedings of the 12th National Conference on Artificial Intelligence (AAAI), 362--367, 1994.}


\end{thebibliography}
\end{document}